\documentclass{amsart}
\usepackage[T1]{fontenc}
\usepackage{mathpazo}          
\usepackage[letterpaper, hmargin=1.3in, vmargin=1.35in]{geometry}

\usepackage[protrusion=true, expansion=true]{microtype}

\usepackage{amsmath, amssymb, amsthm}
\usepackage{mathtools}
\mathtoolsset{mathic=true}
\usepackage{newtxmath}
\usepackage{tikz}
\usetikzlibrary{calc}
\usepackage{enumitem}
\usepackage{mathrsfs}
\usepackage{bbm}
\usepackage{eucal}
\usepackage{bookmark}
\usepackage{keytheorems}

\newkeytheorem{theorem}[parent=section]
\newkeytheorem{corollary}[sibling=theorem]
\newkeytheorem{lemma}[sibling=theorem]
\newkeytheorem{question}[sibling=theorem]
\newkeytheorem{proposition}[sibling=theorem]
\newkeytheorem{definition}[sibling=theorem, style=definition]
\newkeytheorem{remark}[sibling=theorem, style=definition]

\newcommand{\N}{\mathbb{N}} 
\newcommand{\Z}{\mathbb{Z}} 
\newcommand{\R}{\mathbb{R}} 
\newcommand{\C}{\mathbb{C}} 
\DeclareMathOperator{\inter}{\mathrm{int}}

\newcommand{\effros}[1]{\mathcal{F}(#1)}
\newcommand{\effrosdis}[1]{\mathcal{F}_{\textrm{dis}}(#1)}
\newcommand{\vietoris}[2][]{\mathcal{K}_{#1}(#2)}
\newcommand{\holom}[1][]{\mathcal{E}_{#1}}

\newcommand{\ran}{\textbf{ran}}

\newcommand{\acts}{\curvearrowright}
\newcommand{\norm}[1]{\lVert#1\rVert}
\newcommand{\abs}[1]{\lvert#1\rvert}

\title{Separating Orbits by Entire Functions}
\author{Billy Duckworth}
\address{Department of Mathematics, Iowa State University, Ames, IA 50011, USA}

\author{Konstantin Slutsky}
\address{Department of Mathematics, Iowa State University, Ames, IA 50011, USA}
\thanks{This work was partially supported by NSF grant DMS-2153981.}

\begin{document}

\begin{abstract}
  We show that for any free probability measure-preserving action of \(\C^{d}\) on a standard
  probability space, there exists a Borel entire function \(F\) such that the factor map
  \(x \mapsto F_{x}\), where \(F_{x}(z) = F(z \cdot x)\), is injective.  This work builds on a result
  of Gl\"ucksam and Weiss, who constructed non-constant measurable entire functions for such actions.
  The proof combines a separating cross-section whose cocycle values lie in a countable subgroup with
  Forstneri{\v c}'s holomorphic approximation theorem with prescribed critical points.
\end{abstract}

\maketitle

\section{Introduction}
\label{sec:introduction}

Denote by \(\holom[d]\) the Fr\'echet space of entire functions \(\C^{d} \to \C\) equipped with the
topology of locally uniform convergence.  The group \(\C^{d}\) acts on \(\holom[d]\) by argument
shifts: \([z \cdot f](w) = f(w + z)\).  Let \(\C^{d} \acts X\) be a free Borel action of \(\C^{d}\)
on a standard Borel space.  A \emph{Borel entire function} for the action \(\C^{d} \acts X\) is a
Borel map \(F : X \to \C\) such that the function \(F_{x} : \C^{d} \to \C\) defined by
\(F_{x}(z) = F(z \cdot x)\) is entire for every \(x \in X\).  Equivalently, it is a Borel
equivariant map \(x \mapsto F_{x}\) from \(\C^{d} \acts X\) to \(\C^{d} \acts \holom[d]\).  When
\(X\) is given a structure of a standard probability space and the action \(\C^{d} \acts X\) is
measure-preserving, a \emph{measurable entire function} is a Borel function \(F : X \to \C\) that is
Borel entire when restricted to an invariant Borel subset of full measure.

In a recent work, Gl\"ucksam and Weiss~\cite{glucksam2025} constructed non-constant measurable
entire functions for free probability measure-preserving actions of \(\C^{d}\) on standard
probability spaces.  Their construction is based on an inductive approximation over polynomially
convex regions within orbits, using the Oka--Weil theorem at each step.

The purpose of this note is to prove the existence of separating measurable entire functions, i.e.,
functions \(F\) for which the factor map \(x \mapsto F_{x}\) is \emph{injective}.

\getkeytheorem{mainthm}

Recall that a compact set \(K \subseteq \C^{d}\) is \emph{polynomially convex} if its polynomial hull
\[\hat{K} = \{z \in \C^{d} : |p(z)| \le \sup_{K} |p| \text{ for all polynomials } p\}\] equals
\(K\).  A \emph{polynomially convex Borel toast} is a nested system of regions within orbits, formalized in
Definition~\ref{def:toast} and Definition~\ref{def:polynomially-convex-toast} below.  A key result
of Gl\"ucksam and Weiss is the existence of such toasts on a set of full measure for any free
probability measure-preserving action of \(\C^{d}\) (cf.~Theorem~\ref{thm:glucksam-weiss}).
Combined with Theorem~\ref{thm:injective-measurable-entire}, this leads to the following.

\getkeytheorem{maincor}

The key new ingredient is the construction of \emph{separating cross-sections} with cocycle values
in a prescribed countable subgroup.  Recall that a Borel \emph{cross-section} for a free Borel
action \(G \acts X\) of a locally compact second countable group is a Borel set
\(\mathcal{C} \subseteq X\) that meets every orbit in a discrete set.  With a cross-section
\(\mathcal{C}\) one associates the equivariant map \(\phi : X \to \effrosdis{G}\) into the space of
discrete subsets of~\(G\), given by \(\phi(x) = \{g \in G : g^{-1}x \in \mathcal{C}\}\).  A
cross-section is \emph{separating} if \(\phi\) is injective.  The \emph{cocycle} of the action is
the map~\(\rho : E_{G} \to G\) satisfying \(\rho(x,y) \cdot x = y\) for orbit-equivalent
\(x, y \in X\), i.e., \(x E_{G} y\).  Let \(E_{\mathcal{C}}\) denote the restriction of the orbit
equivalence relation onto \(\mathcal{C}\).

\getkeytheorem{mainlemma}

Separating cross-sections with cocycle values in a countable subgroup may find other applications
for establishing injectivity of equivariant maps into various function spaces.  The idea of the
proof of Theorem~\ref{thm:injective-measurable-entire} is to build an entire function whose critical
points form a separating cross-section.  The approximation tool that makes this possible is a
theorem of Forstneri{\v c}~\cite{forstneric2003}, which allows holomorphic approximation on
polynomially convex sets while prescribing the exact set of critical points.

The paper is organized as follows.  Section~\ref{sec:inject-rati-cross} constructs separating
cross-sections for free Borel actions of non-discrete non-compact locally compact second countable
groups.  Section~\ref{sec:polyn-conv-borel} reviews the notion of a Borel toast and explains the
construction of polynomially convex Borel toasts from~\cite{glucksam2025}.
Section~\ref{sec:inject-meas-entire-funct} combines these ingredients to prove
Theorem~\ref{thm:injective-measurable-entire}.

\medskip\noindent\textbf{Use of AI tools.}  AI writing assistants were used during the preparation of
this paper to aid with proofreading, editing, diagram creation, literature review, and content
critique.  All core mathematical ideas are due to the authors.

\medskip\noindent\textbf{Acknowledgments.}  This project has benefited greatly from the
work~\cite{ssw} done by M.~Sodin and A.~Wennman in collaboration with the second author.  We are
grateful to them for numerous helpful and productive conversations.

\section{Separating Rational Cross-Sections}
\label{sec:inject-rati-cross}

Let \(G\) be a locally compact second countable group.  The space \(\effrosdis{G}\) of discrete
subsets of \(G\) is naturally a subset of the Effros Borel space \(\effros{G}\) of all closed
subsets of \(G\).  Since \(\effrosdis{G}\) is Borel in \(\effros{G}\) (see, e.g.,
\cite[Sec.~5.6]{ssw}), it inherits the structure of a standard Borel space.  The group \(G\) acts on
\(\effrosdis{G}\) by left shifts, \((g, F) \mapsto gF\).

Suppose that \(G \acts X\) is a free Borel action on a standard Borel space.  Let \(E_{G}\) denote
the orbit equivalence relation of this action and let \(\rho : E_{G} \to G\) be the cocycle defined
by \(\rho(x,y) \cdot x = y\). A Borel \emph{cross-section} is a Borel set \(\mathcal{C} \subseteq X\)
such that \(\{g \in G : g^{-1}x \in \mathcal{C}\}\) is discrete in \(G\) for every \(x \in X\).
Given \(U \subseteq G\), a cross-section \(\mathcal{C}\) is \emph{\(U\)-lacunary} if
\((U \cdot x) \cap (U \cdot y) = \varnothing\) for all distinct \(x, y \in \mathcal{C}\).  We let
\(E_{\mathcal{C}}\) denote the restriction of \(E_{G}\) to \(\mathcal{C}\).

With a cross-section \(\mathcal{C}\), we associate the map \(\phi: X \to \effrosdis{G}\) given by
\[\phi(x) = \{g \in G: g^{-1}x \in \mathcal{C}\}.\]
Since \(\phi(hx) = h\phi(x)\), the map \(\phi\) is equivariant with respect to the left shift
\(G \acts \effrosdis{G}\).  A cross-section is said to be \emph{separating} if the corresponding
\(\phi\) is injective.

For a free action of a non-discrete locally compact second countable group \(G\) on a standard Borel
space, separating cross-sections are easy to build.  By Kechris's theorem~\cite{kechris1992}, for
any precompact neighborhood of the identity \(U \subseteq G\), there exists a \(U^{2}\)-lacunary
cross-section \(\mathcal{C} \subseteq X\).  Since \(G\) is non-discrete, \(U\) is uncountable, and
we may pick a Borel injection \(\omega : \mathcal{C} \to U \setminus \{e_{G}\}\).  The set
\[\mathcal{D} = \mathcal{C} \cup \{\omega(x) \cdot x : x \in \mathcal{C}\}\]
is then a separating cross-section: each orbit is uniquely determined by any cocycle value
\(\rho(x,y) \in U \setminus \{e_{G}\}\). Taking \(\omega : \mathcal{C} \to U
\setminus V^{2}\) for a small neighborhood \(V\) further ensures that \(\mathcal{D}\) is
\(V\)-lacunary.

In the Borel dynamics of \(\R^{d}\)-flows, it is convenient to work with cross-sections on which
\(\rho\) takes only countably many values, as this significantly simplifies the verification of
Borel measurability. However, the separating cross-section \(\mathcal{D}\) constructed above carries
uncountably many cocycle values, even when the original \(\mathcal{C}\) does not.  The purpose of
this section is to show that both properties can be achieved simultaneously.

\begin{lemma}[store=mainlemma]
  \label{lem:injective-cross-sections}
  Let \(G \acts X\) be a free Borel action of a non-discrete non-compact locally compact second
  countable group on a standard Borel space.  Let \(\Gamma \le G\) be a countable dense subgroup and
  let \(\mathcal{C}\) be a lacunary cross-section such that
  \(\rho(E_{\mathcal{C}}) \subseteq \Gamma\).  There exists a lacunary Borel separating cross-section
  \(\mathcal{D} \supseteq \mathcal{C}\) such that \(\rho(E_{\mathcal{D}}) \subseteq \Gamma\).
\end{lemma}

\begin{figure}[ht]
\centering
\begin{tikzpicture}[
    cpt/.style={circle, fill, inner sep=0pt, minimum size=4.5pt},
    dpt/.style={circle, draw, semithick, fill=white, inner sep=0pt, minimum size=4.5pt},
    >=stealth,
]

\begin{scope}

  \draw[rounded corners=14pt, densely dashed, gray!70!black]
    (-0.4, -1.1) rectangle (5.3, 1.5);
  \node[gray!70!black, font=\scriptsize] at (2.45, -1.4) {$F_n$-class};

  \node[cpt, label={[font=\footnotesize]below:$x_1$}] (c1) at (0.4, 0.6) {};
  \node[cpt, label={[font=\footnotesize]below:$x_2$}] (c2) at (1.8, -0.5) {};
  \node[cpt, label={[font=\footnotesize]below:$x$}] (cx) at (3.2, 0.3) {};
  \node[cpt, label={[font=\footnotesize]below:$x_4$}] (c4) at (4.6, 0.7) {};

  \node[cpt] at (-1.6, -0.3) {};
  \node[cpt] at (6.5, 0.5) {};

  \draw[gray, densely dashed] (cx) circle (0.5cm);

  \coordinate (ztl1) at (8, 2.1);
  \draw[rounded corners=5pt, semithick] (ztl1) rectangle (10.8, -0.6);

  \draw[gray, thin] ($(cx)+(0.35,0.35)$) -- (ztl1);
  \draw[gray, thin] ($(cx)+(0.35,-0.35)$) -- (8,-0.6);

  \node[cpt, label={[font=\footnotesize]below:$x$}] (zx1) at (9.4, 0) {};
  \node[dpt, label={[font=\footnotesize]right:{$\gamma \cdot x$}}] (zgx) at (9.4, 1.6) {};
  \draw[->, semithick] (zx1) -- node[right, font=\footnotesize] {$\gamma$} (zgx);

  \node[font=\bfseries\footnotesize] at (4.5, -2.1) {Case 1 (non-Boolean)};

\end{scope}

\begin{scope}[yshift=-6cm]

  \draw[rounded corners=14pt, densely dashed, gray!70!black]
    (-0.4, -1.1) rectangle (5.3, 1.5);
  \node[gray!70!black, font=\scriptsize] at (2.45, -1.4) {$F_n$-class};

  \node[cpt, label={[font=\footnotesize]below:$x_1$}] (d1) at (0.4, 0.6) {};
  \node[cpt, label={[font=\footnotesize]below:$x_2$}] (d2) at (1.8, -0.5) {};
  \node[cpt, label={[font=\footnotesize]below:$x$}] (dx) at (3.2, 0.3) {};
  \node[cpt, label={[font=\footnotesize]below:$x_4$}] (d4) at (4.6, 0.7) {};

  \node[cpt] at (-1.6, -0.3) {};
  \node[cpt] at (6.5, 0.5) {};

  \draw[gray, densely dashed] (dx) circle (0.5cm);

  \coordinate (ztl2) at (8, 2.5);
  \draw[rounded corners=5pt, semithick] ($(ztl2)+(-0.2,0.2)$) rectangle (11.0, -0.6);

  \draw[gray, thin] ($(dx)+(0.35,0.35)$) -- ($(ztl2)+(-0.2,0.2)$);
  \draw[gray, thin] ($(dx)+(0.35,-0.35)$) -- (7.8,-0.6);

  \node[cpt, label={[font=\footnotesize]below:$x$}] (zx2) at (9.4, 0) {};
  \node[dpt, label={[font=\footnotesize]above:{$g_0\gamma_k x$}}] (zg0) at (8.4, 2) {};
  \node[dpt, label={[font=\footnotesize]above:{$g_1\gamma_k x$}}] (zg1) at (10.4, 2) {};
  \draw[->, semithick] (zx2) --
    node[left, font=\footnotesize, pos=0.35] {$g_0\gamma_k$} (zg0);
  \draw[->, semithick] (zx2) --
    node[right, font=\footnotesize, pos=0.35] {$g_1\gamma_k$} (zg1);
  \draw[densely dashed, semithick] (zg0) --
    node[above, font=\footnotesize] {$g_0 g_1$} (zg1);

  \node[font=\bfseries\footnotesize] at (4.5, -2.1) {Case 2 (Boolean)};

\end{scope}

\end{tikzpicture}
\caption{Construction of the separating cross-section \(\mathcal{D}\).  Filled dots
  (\(\bullet\)) represent points of \(\mathcal{C}\); hollow dots (\(\circ\)) represent
  adjoined points in \(\mathcal{D} \setminus \mathcal{C}\).  The dashed circle marks the
  neighborhood of the representative \(x = s_n(x)\), shown magnified to the right.  In
  Case~1, a single point \(\gamma \cdot x\) is added, where \(\gamma \in \Delta\) encodes
  the relative positions of the class members and their labels.  In Case~2, two points
  \(g_0\gamma_k x\) and \(g_1\gamma_k x\) are added, forming a triangle with~\(x\); the
  index~\(k\) carries the same encoding.}
\label{fig:cross-section-construction}
\end{figure}
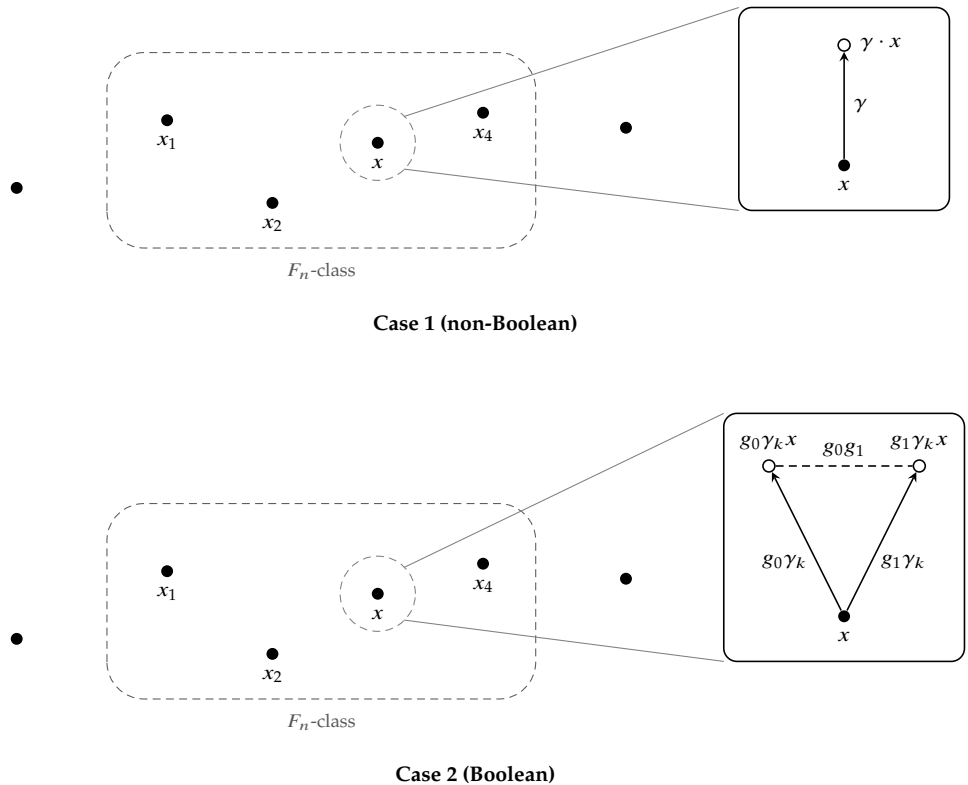

\begin{proof}
  Let \(\norm{\cdot}\) be a proper norm on \(G\), which exists by~\cite{struble1974}, and let
  \(B_{r}\) denote the closed ball of radius~\(r\).  Rescaling the norm, we may assume that
  \(\mathcal{C}\) is \(B_{2}\)-lacunary.  The construction of \(\mathcal{D}\) depends on whether
  \(G\) has an open Boolean subgroup.  Recall that a group \(H\) is \emph{Boolean} if
  \(h^{2} = e_{H}\) for all \(h \in H\); such a group is necessarily Abelian and can be viewed as a
  vector space over \(\Z/2\Z\).

  \textbf{Case 1:} Suppose that \(G\) does not have an open Boolean subgroup.  Then every
  neighborhood of the identity contains infinitely many elements of order greater than two.  In
  fact,
  there exists \(\delta > 0\) for which the set
  \begin{equation}
    \label{eq:1}
     \{g \in G : g^{2} \ne e_{G},\; 2\delta < \norm{g} < 1\}
  \end{equation}
  is infinite.  Since the condition in~\eqref{eq:1} is open and \(\Gamma\) is dense in \(G\), the
  set \eqref{eq:1} contains infinitely many elements of \(\Gamma\).  We can therefore choose an
  infinite set \(\Delta \subseteq \Gamma\) contained in the set \eqref{eq:1} and satisfying
  \(\Delta \cap \Delta^{-1} = \varnothing\).

  We construct \(\mathcal{D}\) by adjoining to \(\mathcal{C}\) points of the form \(\gamma \cdot x\)
  for selected \(x \in \mathcal{C}\) and \(\gamma \in \Delta\).  The norm bounds on elements of
  \(\Delta\) guarantee that \(\mathcal{D}\) is a \(B_{\delta}\)-lacunary cross-section.  Moreover,
  for any distinct \(x, y \in \mathcal{D}\) with \(\norm{\rho(x,y)} < 1\), one of the two points,
  say \(x\),
  lies in \(\mathcal{C}\) and the other equals \(\gamma \cdot x\) with \(\gamma = \rho(x,y) \in \Delta\).
  Since \(\rho(y,x) = \rho(x,y)^{-1}\) and \(\Delta \cap \Delta^{-1} = \varnothing\), the cocycle
  value \(\rho(x,y)\) alone determines which of \(x\), \(y\) belongs to \(\mathcal{C}\).  In the
  Boolean case treated below, this distinction cannot be read off from \(\rho(x,y)\), and we encode
  it differently.

  Since \(G\) is non-compact, we may assume without loss of generality that \(E_{\mathcal{C}}\) is
  aperiodic, and we may choose an aperiodic hyperfinite subrelation
  \(E \subseteq E_{\mathcal{C}}\),~\cite[Lem.~3.25]{jackson2002}.  Let \((F_{n})_{n \ge 0}\) be an
  increasing sequence of finite Borel equivalence relations on \(\mathcal{C}\) with
  \(E = \bigcup_{n} F_{n}\).  By~\cite[Cor.~9.7]{djk1994}, we may arrange that \(F_{0}\) is the
  equality relation and each \(F_{n}\)-class, \(n \ge 1\), is the union of exactly two
  \(F_{n-1}\)-classes.  In particular, every \(F_{n}\)-class has cardinality \(2^{n}\).

  Next, we build Borel selectors \(s_{n} : \mathcal{C} \to \mathcal{C}\), \(n \ge 1\), for the
  relations \(F_{n}\) with pairwise disjoint images.  Fix a Borel linear order
  \(\prec_{\mathcal{C}}\) on \(\mathcal{C}\) and let \(s_{1}(x)\) be the
  \(\prec_{\mathcal{C}}\)-least element of the \(F_{1}\)-class of \(x\).  For \(s_{2}\), note that
  each \(F_{2}\)-class consists of two \(F_{1}\)-classes, each of size two, so it contains two
  elements outside the range of \(s_{1}\).  We set \(s_{2}(x)\) to be the
  \(\prec_{\mathcal{C}}\)-least element of the \(F_{2}\)-class of \(x\) not in the range of
  \(s_{1}\).  Continuing inductively, we obtain Borel maps \(s_{n} : \mathcal{C} \to \mathcal{C}\),
  \(n \ge 1\), satisfying \(s_{n}(x) F_{n} x\) for all \(x \in \mathcal{C}\) and having pairwise
  disjoint ranges \(s_{n}(\mathcal{C})\).

  Pick a Borel injection \(\alpha : \mathcal{C} \to 2^{\N}\) and set
  \(\alpha_{n}(x) := \alpha(x)|_{n}\).  Pick injections
  \[\upsilon_{n} : (\Gamma \times 2^{n})^{2^{n}} \to \Delta, \quad n \ge 1,\]
  with pairwise disjoint images.  For \(x \in \mathcal{C}\) and \(n \ge 1\), write \(x_{1}, \ldots,
  x_{2^{n}}\) for the elements of the \(F_{n}\)-class of \(x\) enumerated in increasing
  \(\prec_{\mathcal{C}}\)-order, and define \(\beta_{n} : \mathcal{C} \to (\Gamma \times
  2^{n})^{2^{n}}\) by \[\beta_{n}(x) = \bigl(\rho(s_{n}(x),x_{i}),
    \alpha_{n}(x_{i})\bigr)_{i=1}^{2^{n}}.\]
  The tuple \(\beta_{n}(x)\) thus records, for each member \(x_{i}\) of the \(F_{n}\)-class of
  \(x\), its position relative to \(s_{n}(x)\) together with the first \(n\) digits of
  \(\alpha(x_{i})\). The required cross-section is
  \[\mathcal{D} = \mathcal{C} \cup \bigcup_{n \ge 1} \{\upsilon_{n}(\beta_{n}(x)) \cdot x :
    x \in s_{n}(\mathcal{C})\}.\]
  In other words, for each \(n \ge 1\) and each representative \(x \in s_{n}(\mathcal{C})\), we
  adjoin the point \(\gamma \cdot x\) where \(\gamma = \upsilon_{n}(\beta_{n}(x))\), so that the
  cocycle value \(\rho(x, \gamma \cdot x)\) encodes the structure of the \(F_{n}\)-class of \(x\)
  and the first \(n\) digits of its members' labels.

  We claim that \(\mathcal{D}\) satisfies the conclusion of the lemma.  It is evidently Borel, and
  the cocycle takes values in the countable group \(\Gamma\).  Lacunarity is
  preserved because \(2\delta < \norm{\rho(x, \gamma \cdot x)} < 1\) for all \(\gamma \in \Delta\),
  while \(\norm{\rho(x,y)} > 2\) for distinct \(x, y \in \mathcal{C}\).  It remains to verify
  injectivity.

  Suppose for contradiction that \(\phi\) is not injective, and let \(x, y \in X\) be distinct
  points with \(\phi(x) = \phi(y)\), where \(\phi : X \to \effrosdis{G}\) is the equivariant map
  associated with \(\mathcal{D}\).  By equivariance, we may replace \(x\) and \(y\) by \(gx\) and
  \(gy\) for any \(g \in G\), so assume \(x \in \mathcal{C}\).  Since \(e_{G} \in \phi(x)\) if and
  only if \(x \in \mathcal{D}\), we conclude that \(y \in \mathcal{D}\).

  Choose \(n\) large enough that \(\alpha_{n}(x) \ne \alpha_{n}(y)\), and let \(x_{0} =
  s_{n}(x)\).  Setting \(\gamma = \upsilon_{n}(\beta_{n}(x_{0}))\), we have \(\gamma \cdot x_{0}
  \in \mathcal{D}\) by construction, so \(\gamma^{-1} \in \phi(x_{0}) = \rho(x,
  x_{0})\phi(x)\).  Since \(\phi(x) = \phi(y)\), also \(\gamma^{-1} \in
  \rho(x,x_{0})\phi(y)\).  Put \(y_{0} = \rho(x,x_{0}) \cdot y\); then both \(y_{0}\) and
  \(\gamma \cdot y_{0}\) lie in \(\mathcal{D}\), with \(\gamma \in \Delta\).  By construction
  of \(\mathcal{D}\), this forces \(y_{0} \in \mathcal{C}\) and \(y_{0} = s_{n}(y_{0})\).
  (Here we use \(\gamma \ne \gamma^{-1}\): otherwise \(\gamma \cdot y_{0}\) could be the
  representative instead.)  The \(F_{n}\)-classes of \(x_{0}\) and \(y_{0}\) therefore have
  identical relative positions and matching labels up to the first \(n\) digits.  In particular,
  since \(\gamma\) encodes \(\rho(x_{0}, x)\) among the positions in the class of \(x_{0}\), the
  corresponding element \(\rho(x_{0}, x) \cdot y_{0}\) belongs to the \(F_{n}\)-class of
  \(y_{0}\).  But \(\rho(x_{0}, x) \cdot y_{0} = y\), so \(y_{0}\) is the representative of the
  \(F_{n}\)-class of \(y\), and the matching of labels gives \(\alpha_{n}(x) = \alpha_{n}(y)\),
  contradicting the choice of \(n\).

  \textbf{Case 2:} Suppose that \(G\) has an open Boolean subgroup.  Choose a precompact neighborhood \(U\) of
  the identity whose non-trivial elements all have order two, and pick a linearly independent
  sequence \((\gamma'_{n})_{n}\) in \(U\).  By precompactness, we may pass to a Cauchy
  subsequence.  The sequence \((\gamma_{n})_{n}\) defined by \(\gamma_{n} =
  \gamma'_{2n}\gamma'_{2n+1}\) remains linearly independent and satisfies \(\gamma_{n} \to e_{G}\).
  After one further passage to a subsequence, we may assume that there are two distinct elements \(g_{0},
  g_{1} \in U\) such that
  \begin{equation}
    \label{eq:2}
    2\norm{\gamma_{k}} < \min\{\norm{g_{0}}, \norm{g_{1}}, \norm{g_{0}g_{1}}\}
  \end{equation}
  and the elements \(g_{0},g_{1}, \gamma_{0}, \gamma_{1}, \ldots\) are linearly independent.  We
  construct \(\mathcal{D}\) by adjoining to certain representatives \(x \in \mathcal{C}\) the pair
  of points \(g_{0}\gamma_{k} x\) and \(g_{1}\gamma_{k} x\) for an appropriate index \(k\).
  Condition~\eqref{eq:2} ensures that the triple \(x, g_{0}\gamma_{k} x, g_{1}\gamma_{k} x\) is
  uniformly separated, preserving lacunarity.

  Formally, pick injections \(u_{n} : (\Gamma \times 2^{n})^{2^{n}} \to \N\) with
  pairwise disjoint images and set \(\upsilon_{n}(z) = \gamma_{u_{n}(z)}\).  With \(\beta_{n}\)
  defined as before, the cross-section is
  \[\mathcal{D} = \mathcal{C} \cup \bigcup_{n \ge 1} \{g_{i}\upsilon_{n}(\beta_{n}(x)) \cdot x : x \in
    s_{n}(\mathcal{C}),\; i = 0,1\}.\]
  For each representative \(x \in s_{n}(\mathcal{C})\), we thus add two points
  \(g_{0}\gamma_{k} x\) and \(g_{1}\gamma_{k} x\), where \(k = u_{n}(\beta_{n}(x))\)
  encodes the relative positions of all members of the \(F_{n}\)-class and their labels
  \(\alpha_{n}\).  The verification that \(\mathcal{D}\) satisfies the desired properties is
  analogous to Case~1.
\end{proof}

\section{Polynomially Convex Borel Toasts}
\label{sec:polyn-conv-borel}

In their recent work, Gl\"ucksam and Weiss~\cite{glucksam2025} constructed measurable entire
functions of several complex variables.  More precisely, given a free probability measure-preserving
action \(\C^{d} \acts X\) on a standard probability space \((X,\mu)\), they constructed
Borel\footnote{We cast all definitions in the context of Borel dynamics.}
functions \(F : X \to \C\) such that the maps \(F_{x} : \C^{d} \to \C\) given by
\(F_{x}(z) = F(z \cdot x)\) are non-constant entire functions for almost all \(x \in X\).  When
\(F_{x}\) is entire for all \(x \in X\), we say that \(F\) is a \emph{Borel entire} function.

Their construction is inductive over polynomially convex regions within orbits.  In Borel dynamics,
the idea of coherent and exhaustive regions is formalized through the notion of a Borel toast.  For
the purpose of this discussion, we adopt the following version of this concept
from~\cite[Def.~3.2]{ssw}.  Let \(\vietoris[*]{\C^{d}}\) denote the Vietoris space of compact
subsets of \(\C^{d}\) with non-empty interior.

\begin{definition}
  \label{def:toast}
  Let \(\C^{d} \acts X\) be a free Borel action on a standard Borel space. A Borel \emph{toast} for
  the action is a sequence \((\mathcal{C}_{n}, \lambda_{n})_{n}\) of cross-sections
  \(\mathcal{C}_{n} \subseteq X\) and Borel functions
  \(\lambda_{n} : \mathcal{C}_{n} \to \vietoris[*]{\C^{d}}\) satisfying the following
  conditions. For each \(n\), define \(R_{n}(c) = \lambda_{n}(c)\cdot c\) for
  \(c \in \mathcal{C}_{n}\), and let \(X_{n} = \bigcup_{c \in \mathcal{C}_{n}} R_{n}(c)\).
  \begin{enumerate}
  \item\label{item:toast-disjoint} \(R_{n}(c_{n}) \cap R_{n}(c'_{n}) = \varnothing\) for all
    distinct \(c_{n}, c'_{n} \in \mathcal{C}_{n}\).
  \item\label{item:toast-coherent} For all \(m < n\), \(c_{m} \in \mathcal{C}_{m}\), and
    \(c_{n} \in \mathcal{C}_{n}\), either \(R_{m}(c_{m}) \cap R_{n}(c_{n}) = \varnothing\) or
    \(R_{m}(c_{m}) \subseteq R_{n}(c_{n})\).
  \item\label{item:toast-layered} For all \(c_{n} \in \mathcal{C}_{n}\) there exists
    \(c_{n+1} \in \mathcal{C}_{n+1}\) such that
    \(R_{n}(c_{n}) \subseteq R_{n+1}(c_{n+1})\)\footnote{This property should not be confused with a
      stronger condition in Borel dynamics of \(\Z^{d}\)-actions, sometimes called
      \emph{layeredness}.  We allow regions \(R_{n-1}(c_{n-1})\) and \(R_{n}(c_{n})\) to coincide,
      and this item can be easily achieved by re-indexing the points if necessary.}.
  \item\label{item:toast-directed} For all \(c_{m_{1}} \in \mathcal{C}_{m_{1}}\) and
    \(c_{m_{2}} \in \mathcal{C}_{m_{2}}\) satisfying \(c_{m_{1}} E c_{m_{2}}\), there exist
    \(n > m_{1}, m_{2}\) and an element \(c_{n} \in \mathcal{C}_{n}\) such that
    \(R_{m_{i}}(c_{m_{i}}) \subseteq \inter R_{n}(c_{n})\) for \(i = 1,2\), where
    \(\inter R_{n}(c_{n}) = (\inter \lambda_{n}(c_{n})) \cdot c_{n}\).
  \item\label{item:toast-lacunary} There exists a neighborhood of the origin \(U \subseteq \C^{d}\) such that
    \(U \subseteq \lambda_{n}(c_{n})\) for all \(c_{n} \in \mathcal{C}_{n}\) and all \(n\).
  \item\label{item:toast-exhaustive} \(\bigcup_{n} X_{n} = X\).
  \item\label{item:toast-rational} The range \(\ran \lambda_{n}\) is countable for each \(n\) and
    \(\{\rho(c_{m},c_{n}) : c_{m} \in \mathcal{C}_{m}, c_{n} \in \mathcal{C}_{n}, c_{m} E
      c_{n}\}\)
    is countable.
  \end{enumerate}
  With a toast \((\mathcal{C}_{n}, \lambda_{n})_{n}\) we associate maps
  \(\pi_{n} : X_{n} \to \mathcal{C}_{n}\) given by the condition \(x \in R_{n}(\pi_{n}(x))\).
\end{definition}

Borel toasts help in constructing a function \(F : X \to \C\) as a limit of functions
\(F_{n} : X_{n} \to \C\).  During the inductive step of the construction, one considers a region
\(R_{n}(c_{n})\), \(c_{n} \in \mathcal{C}_{n}\), and all the (pairwise disjoint) regions
\(R_{n-1}(c^{i}_{n-1}) \subseteq R_{n}(c_{n})\) for \(c_{n-1}^{i} \in \mathcal{C}_{n-1}\),
\(1 \le i \le m\).  The function \(F_{n-1}\) is defined on each \(R_{n-1}(c_{n-1}^{i})\) in the
previous step of the construction.  One picks an approximate extension of \(F_{n-1}\), i.e., an
extension to a function \(F_{n} : X_{n} \to \C\) such that
\(\abs{F_{n-1}(x) - F_{n}(x)} < \epsilon\) for all \(x\) in the regions \(R_{n-1}(c_{n-1}^{i})\) and
a suitably chosen \(\epsilon > 0\).

Assumption~\eqref{item:toast-rational} in the definition of a toast serves primarily to simplify the
verification of Borel measurability of the functions \(F_{n}\) and the limit function \(F\).  More
specifically, it ensures that there are only countably many possible shapes of regions
\(R_{n}(c_{n})\) and countably many different configurations of regions \(R_{n-1}(c_{n-1}^{i})\)
inside each of them.  So, if the inductive assumption tells us that the value of \(F_{n-1}\)
depends, for instance, only on the shape of these regions, we can partition
\(\mathcal{C}_{n} = \bigsqcup_{k} \mathcal{C}_{n,k}\) into countably many pieces with the same
region configuration and define \(F_{n}\) on each piece of the partition separately by
\(F_{n}(x) = f_{n,k}(\rho(\pi_{n}(x), x))\) for \(x \in \pi_{n}^{-1}(\mathcal{C}_{n,k})\), where
\(f_{n,k} : K_{n,k} \to \C\) is any measurable function that achieves the desired precision
of approximation over \(K_{n,k} = \lambda_{n}(c_{n})\), \(c_{n} \in \mathcal{C}_{n,k}\).

In the application discussed in Section~\ref{sec:inject-meas-entire-funct}, the values of
\(F_{n-1}\) will not be determined solely by the configurations of the regions \(R_{n}(c_{n})\) and
sub-regions inside them, but they will still be determined by a countable amount of data, where the
new data will be a separating cross-section as in Lemma~\ref{lem:injective-cross-sections}.  The
Borel measurability of the limit function \(F\) is automatic for the same reason.

The second key ingredient needed to run the construction is the ability to choose the desired
approximation function \(f_{n,k}\).  For the construction of measurable entire functions of several
complex variables, this boils down to the following problem.  Given a region
\(\lambda_{n}(c_{n}) \subseteq \C^{d}\), pairwise disjoint sub-regions
\(K_{n-1,i} = \lambda_{n-1}(c_{n-1}^{i}) + \rho(c_{n}, c_{n-1}^{i})\), \(1 \le i \le m\), and
holomorphic functions \(f_{n-1, i} : K_{n-1,i} \to \C\), find a holomorphic function
\(f_{n} : \lambda_{n}(c_{n}) \to \C\) that approximates \(f_{n-1,i}\) uniformly on their
corresponding domains \(K_{n-1,i}\) up to a given, exponentially small, \(\epsilon > 0\).

In the one-dimensional case, \(d=1\), this is possible provided all regions
\(\lambda_{k}(c_{k}) \subseteq \C\) have connected complements; the existence of the desired
approximate extension \(f_{n}\) is guaranteed by Runge's theorem.  In higher dimensions,
\(d \ge 2\), one can use the Oka--Weil theorem, which requires polynomial convexity.  Unlike the
property of having connected complements, polynomial convexity is not preserved under disjoint
unions.  The precise condition needed to run the inductive construction based on the Oka--Weil
theorem is that the disjoint union
\(\bigsqcup_{i=1}^{m}(\lambda_{n-1}(c_{n-1}^{i}) + \rho(c_{n}, c_{n-1}^{i}))\) be polynomially
convex.

\begin{definition}
  \label{def:polynomially-convex-toast}
  We say that a Borel toast \((\mathcal{C}_{n}, \lambda_{n})_{n}\) is \emph{polynomially convex} if
  for all \(c_{n} \in \mathcal{C}_{n}\) the set
  \[\bigsqcup_{i=1}^{m} \bigl(\lambda_{n-1}(c_{n-1}^{i}) + \rho(c_{n}, c_{n-1}^{i})\bigr)\]
  is polynomially convex, where the union is parametrized by subregions \(R_{n-1}(c_{n-1}^{i}) \subseteq
  R_{n}(c_{n})\).
\end{definition}

The following result can be extracted from~\cite{glucksam2025}.

\begin{theorem}[{Gl\"ucksam--Weiss~\cite{glucksam2025}}]
  \label{thm:glucksam-weiss}
  Given a free probability measure-preserving action \(\C^{d} \acts X\), there exists an invariant
  Borel subset of full measure on which there exists a polynomially convex Borel toast.
\end{theorem}

Since the terminology used in~\cite{glucksam2025} is different, we explain the connection between
their work and the formulation in Theorem~\ref{thm:glucksam-weiss}.  Within the context of ergodic
theory, a typical way of constructing a toast for a \(\C^{d}\) (or \(\R^{d}\), for that matter)
action on a subset of full measure goes as follows.  One constructs sets \(Y_{n} \subseteq X\),
\(n \in \N\), such that for almost all \(x \in X\):
\begin{enumerate}[label=(\Alph*), ref=\Alph*]
\item\label{toast-seq-components} Connected components of \(\{z \in \C^{d} : z \cdot x \in Y_{n}\}\)
  are compact for all \(n\).
\item\label{toast-seq-exhaustive} For every compact \(K \subseteq \C^{d}\) there
  exists \(n\) such that \(K \cdot x \subseteq \bigcap_{k \ge n} Y_{k}\).
\end{enumerate}
Connected components of the sets \(Y_{n}\) within orbits, as in item~\eqref{toast-seq-components},
are usually simple sets, often just boxes \(\prod_{i=1}^{2d}[a_{i},b_{i}]\).  Let
\(X_{n} := \bigcap_{k \ge n}Y_{k}\).  One argues that connected components of the sets \(X_{n}\)
essentially form the regions \(R_{n}(c_{n})\) of a toast.
Items~\eqref{item:toast-disjoint}--\eqref{item:toast-layered} are automatic; for
example,~\eqref{item:toast-coherent} is a manifestation of the fact that if \(A \subseteq B\) and
\(C_{A} \subseteq A\), \(C_{B} \subseteq B\) are connected components of \(A\) and \(B\),
respectively, then either \(C_{A} \subseteq C_{B}\) or \(C_{A} \cap C_{B} = \varnothing\).
Items~\eqref{item:toast-directed} and~\eqref{item:toast-exhaustive} use
assumption~\eqref{toast-seq-exhaustive}.  Representatives \(c_{n}\) of the connected components of
\(X_{n}\) could be taken to be the lexicographically smallest elements.  However, to satisfy
item~\eqref{item:toast-lacunary}, one needs to assume that connected components of \(X_{n}\) have
large interior and choose \(c_{n}\) to be an interior point.  Furthermore, the points \(c_{n}\) can
be shifted to ensure that they lie on the rational grid~\cite[Lem.~2.5]{slutsky2019}, i.e.,
\(\rho(c_{m},c_{n})\) has rational coordinates and therefore takes only countably many values.

The only aspect of Definition~\ref{def:toast} that is not straightforwardly related to the viewpoint
of taking \(R_{n}(c_{n})\) to be the connected components of \(X_{n}\) is the requirement that
\[\lambda_{n}(c_{n}) := \{z \in \C^{d}: z\cdot c_{n} \in R_{n}(c_{n})\}\] take only countably many
values.  Since the Vietoris space \(\vietoris[*]{\C^{d}}\) is separable, one can always modify the
shapes of the regions so that they belong to a countable dense subset of the Vietoris space.
However, even a slight change in the shape of each region \(R_{n-1}\) may violate the condition that
\(\bigsqcup_{i=1}^{m} \bigl(\lambda_{n-1}(c_{n-1}^{i}) + \rho(c_{n}, c_{n-1}^{i})\bigr)\) remains
polynomially convex.

\begin{figure}[ht]
\centering
\begin{tikzpicture}[scale=1.1,
    sq/.style={fill=gray!18, draw=black, semithick},
  ]
  \pgfmathsetmacro{\halfth}{0.1}

  \begin{scope}
    \foreach \i in {1,2,3} {
      \draw[densely dashed, gray!55] (\i,0) -- (\i,4);
      \draw[densely dashed, gray!55] (0,\i) -- (4,\i);
    }
    \draw[semithick] (0,0) rectangle (4,4);
    \filldraw[sq] (0.7,1.8) rectangle (1.7,2.8);
    \filldraw[sq] (1.6,0.6) rectangle (2.6,1.6);
    \filldraw[sq] (0.4,0.4) rectangle (1.4,1.4);
    \filldraw[sq] (2.4,2.5) rectangle (3.4,3.5);
    \node at (1.2,2.3) {\(\tilde{R}_{n-1}\)};
    \node at (2.1,1.1) {\(\tilde{R}_{n-1}\)};
    \node at (0.9,0.9) {\(\tilde{R}_{n-1}\)};
    \node at (2.9,3.0) {\(\tilde{R}_{n-1}\)};
    \node at (3.75,0.25) {\(\tilde{R}_{n}\)};
    \node[font=\scriptsize] at (2.0, -0.5) {\((a)\)};
  \end{scope}

  \begin{scope}[xshift=4.85cm]
    \filldraw[sq] (0.7,1.8) rectangle (1.7,2.8);
    \filldraw[sq] (1.6,0.6) rectangle (2.6,1.6);
    \filldraw[sq] (0.4,0.4) rectangle (1.4,1.4);
    \filldraw[sq] (2.4,2.5) rectangle (3.4,3.5);
    \foreach \i in {1,2,3} {
      \fill[white] (\i-\halfth, 0) rectangle (\i+\halfth, 4);
      \fill[white] (0, \i-\halfth) rectangle (4, \i+\halfth);
    }
    \foreach \i in {1,2,3} {
      \draw[densely dashed, gray!55] (\i,0) -- (\i,4);
      \draw[densely dashed, gray!55] (0,\i) -- (4,\i);
    }
    \draw[semithick] (0,0) rectangle (4,4);
    \draw[thin] (1-\halfth, -0.05) -- (1-\halfth, -0.3);
    \draw[thin] (1+\halfth, -0.05) -- (1+\halfth, -0.3);
    \node[font=\scriptsize] at (1, -0.5) {$\theta$};
    \node[font=\scriptsize] at (2.0, -0.5) {\((b)\)};
  \end{scope}
\end{tikzpicture}
\caption{(a): Four unit squares (shaded) inside a \(4 \times 4\) grid.  (b): The same
  configuration with corridors of width~\(\theta\) along the grid lines, splitting each unit square
  into sub-pieces.}
\label{fig:grid-corridors}
\end{figure}
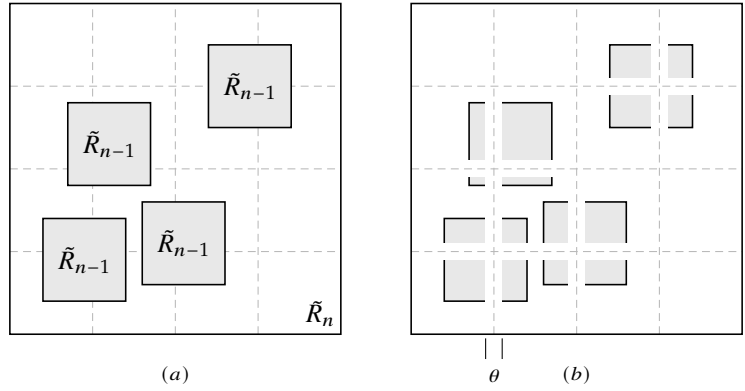

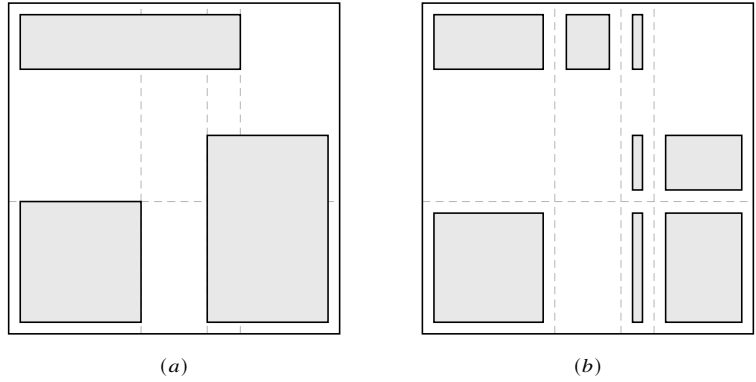
\begin{figure}[ht]
\centering
\begin{tikzpicture}[scale=1.1,
    sq/.style={fill=gray!18, draw=black, semithick},
  ]
  \pgfmathsetmacro{\halfth}{0.1}
  \pgfmathsetmacro{\halfthsm}{0.035}

  \begin{scope}[scale=4]
    \draw[densely dashed, gray!55] (1.4,1.0) -- (1.4,2.0);
    \draw[densely dashed, gray!55] (1.6,1.0) -- (1.6,2.0);
    \draw[densely dashed, gray!55] (1.7,1.0) -- (1.7,2.0);
    \draw[densely dashed, gray!55] (1.0,1.4) -- (2.0,1.4);
    \fill[sq] (1.0+\halfthsm,1.8) rectangle (1.7,2.0-\halfthsm);
    \fill[sq] (1.6,1.0+\halfthsm) rectangle (2.0-\halfthsm,1.6);
    \fill[sq] (1.0+\halfthsm,1.0+\halfthsm) rectangle (1.4,1.4);
    \draw[semithick] (1,1) rectangle (2,2);
    \node[font=\scriptsize] at (1.5, 0.9) {\((a)\)};
  \end{scope}

  \begin{scope}[xshift=5cm, scale=4]
    \fill[sq] (1.0+\halfthsm,1.8) rectangle (1.4-\halfthsm,2.0-\halfthsm);
    \fill[sq] (1.4+\halfthsm,1.8) rectangle (1.6-\halfthsm,2.0-\halfthsm);
    \fill[sq] (1.6+\halfthsm,1.8) rectangle (1.7-\halfthsm,2.0-\halfthsm);
    \fill[sq] (1.6+\halfthsm,1.0+\halfthsm) rectangle (1.7-\halfthsm,1.4-\halfthsm);
    \fill[sq] (1.6+\halfthsm,1.4+\halfthsm) rectangle (1.7-\halfthsm,1.6);
    \fill[sq] (1.7+\halfthsm,1.0+\halfthsm) rectangle (2.0-\halfthsm,1.4-\halfthsm);
    \fill[sq] (1.7+\halfthsm,1.4+\halfthsm) rectangle (2.0-\halfthsm,1.6);
    \fill[sq] (1.0+\halfthsm,1.0+\halfthsm) rectangle (1.4-\halfthsm,1.4-\halfthsm);
    \draw[densely dashed, gray!55] (1.4,1.0) -- (1.4,2.0);
    \draw[densely dashed, gray!55] (1.6,1.0) -- (1.6,2.0);
    \draw[densely dashed, gray!55] (1.7,1.0) -- (1.7,2.0);
    \draw[densely dashed, gray!55] (1.0,1.4) -- (2.0,1.4);
    \draw[semithick] (1,1) rectangle (2,2);
    \node[font=\scriptsize] at (1.5, 0.9) {\((b)\)};
  \end{scope}
\end{tikzpicture}
\caption{(a): Sub-pieces within a grid square. (b): Sub-pieces
  within a grid square cut by face hyperplanes.}
\label{fig:grid-square-corridors}
\end{figure}

The main idea behind the construction of the sets \(Y_{n}\) in~\cite{glucksam2025} is illustrated in
Figures~\ref{fig:grid-corridors} and~\ref{fig:grid-square-corridors}.  Using Rokhlin's lemma, one
starts with nested box regions \(\tilde{R}_{n}\) as in Figure~\ref{fig:grid-corridors}(a).  Each
region is a box whose side length is a multiple of the side length of \(\tilde{R}_{n-1}\).  However,
the regions \(\tilde{R}_{n-1}\) inside an \(\tilde{R}_{n}\) region may not align with the grid
inside \(\tilde{R}_{n}\).  One cuts corridors of width \(\theta_{n-1}\) along the grid hyperplanes
inside \(\tilde{R}_{n}\) as in Figure~\ref{fig:grid-corridors}(b).  Next, consider a box from the
grid and the sub-regions of the \(\tilde{R}_{n-1}\) regions inside it, after removing the corridor
regions, Figure~\ref{fig:grid-square-corridors}(a).  Cut these subregions by corridors of width
\(\theta_{n-1}\) along the hyperplanes determined by the faces of these subregions,
Figure~\ref{fig:grid-square-corridors}(b).  The resulting pieces form the sets \(Y_{n-1}\) in the
notation above.  Each connected component of \(Y_{n-1}\) is a box and is therefore (polynomially)
convex.  Gl\"ucksam and Weiss showed that a union of polynomially convex regions arranged in a (not
necessarily uniform) grid, as in Figure~\ref{fig:grid-square-corridors}(b), remains polynomially
convex~\cite[Cor.~1, p.~8]{glucksam2025}, and deduce that the union of connected components of
\(Y_{n-1}\) inside each region \(\tilde{R}_{n}\) is still polynomially convex.  If we assume that
\(\rho(c_{n},c_{m})\) has rational coordinates whenever \(c_{n}\) and \(c_{m}\) are centers of the
regions \(\tilde{R}_{n}\) and \(\tilde{R}_{m}\) (which can always be arranged
by~\cite[Lem.~2.4]{slutsky2019}), then for each region \(\tilde{R}_{n}\) there are only countably
many possible configurations of subregions.  This ensures that item~\eqref{item:toast-rational} of
Definition~\ref{def:toast} is satisfied.

For the estimates on \(\theta_{n}\) and the verification that the sets \(Y_{n}\) constructed as in
Figure~\ref{fig:grid-corridors} satisfy the required assumptions above, we refer the reader
to~\cite{glucksam2025}.  The verification of item~\eqref{toast-seq-exhaustive} specifically is
based on a Borel--Cantelli-type argument and can be found on p.~15 of~\cite{glucksam2025}.  This
concludes our explanation of how Theorem~\ref{thm:glucksam-weiss} follows from the work of
Gl\"ucksam and Weiss.

\section{Separating Measurable Entire Functions}
\label{sec:inject-meas-entire-funct}

Let \(\C^{d} \acts X\) be a free Borel action. A Borel entire function \(F : X \to \C\) can also be
viewed as a map \( X \ni x \mapsto F_{x} \in \holom[d]\), where \(\holom[d]\) denotes the space of
entire functions of \(d\) complex variables and is endowed with the topology of locally uniform
convergence.  This map is equivariant with respect to the argument shift action
\(\C^{d}\acts \holom[d]\) given by \( [z \cdot f](w) = f(w + z)\).  In other words, a Borel entire
function is a factor map from \(\C^{d} \acts X\) to \(\C^{d} \acts \holom[d]\).  The Oka--Weil
theorem, plugged into the inductive construction over a polynomially convex Borel toast, produces
measurable entire functions.  In this section, we modify this construction and explain how to
produce, for a given probability measure-preserving system \(\C^{d} \acts X\), a measurable entire
function \(F\) for which the factor map \(x \mapsto F_{x}\) is injective.

\begin{definition}
  \label{def:separating-measurable-entire}
  A measurable entire function \(F : X \to \C\) is \emph{separating} if the associated factor map
  \(x \mapsto F_{x}\) is injective on an invariant subset of full measure.  Similarly, a Borel
  entire function is \emph{separating} if the factor map is injective on all of \(X\).
\end{definition}

The idea is to pick a separating cross-section \(\mathcal{D}\) and build a measurable entire
function whose critical points are exactly the elements of the cross-section.  Recall that a
critical point of a function \(f \in \holom[d]\) is a complex vector \(z \in \C^{d}\) such that
\([\nabla f](z) = 0\).  A point \(x \in X\) is critical for \(F : X \to \C\) if \(0\) is
a critical point of \(F_{x}\).  We replace the Oka--Weil theorem with an approximation result that
takes critical points into account.  The following is a special case of a theorem of Forstneri{\v
  c}~\cite{forstneric2003}.

\begin{theorem}[{cf.~\cite[Thm.~2.1]{forstneric2003}}]
  \label{thm:forstneric}
  Let \(K \subseteq \C^{d}\) be a polynomially convex compact set and let \(f\) be a function
  holomorphic in a neighborhood of \(K\) with a finite set of critical points \(P \subseteq K\).  For any
  \(\epsilon > 0\), there
  exists an entire function \(\tilde{f} : \C^{d} \to \C\) whose critical points are exactly \(P\)
  and that satisfies
  \[\sup_{z \in K} |f(z) - \tilde{f}(z)| < \epsilon.\]
\end{theorem}

\begin{theorem}[store=mainthm]
  \label{thm:injective-measurable-entire}
  For any free Borel action \(\C^{d} \acts X\) that admits a polynomially convex Borel toast,
  there exists a separating Borel entire function.
\end{theorem}

\begin{proof}
  Let \((\mathcal{C}_{n}, \lambda_{n})_{n}\) be a polynomially convex Borel toast.  Apply
  Lemma~\ref{lem:injective-cross-sections} to the cross-section \(\mathcal{C}_{0}\) to obtain
  a separating cross-section \(\mathcal{D}\).  We may assume without loss of generality that
  \(\mathcal{D} \subseteq X_{0}\), i.e., that the added points of \(\mathcal{D}\) are sufficiently close
  to points of \(\mathcal{C}_{0}\).

  We now proceed with an inductive construction as described in Section~\ref{sec:polyn-conv-borel}.  More precisely, we partition
  \(\mathcal{C}_{0} = \bigsqcup_{k}\mathcal{C}_{0,k}\) based on the shape of regions
  \(\lambda_{0}(c_{0})\) and the location of points from \(\mathcal{D}\) inside them.  In other
  words, if \(c_{0}, c'_{0} \in \mathcal{C}_{0,k}\) then
  \(\lambda_{0}(c_{0}) = \lambda_{0}(c'_{0})\) and \(z \cdot c_{0} \in \mathcal{D}\) if and only if
  \(z \cdot c'_{0} \in \mathcal{D}\) for all \(z \in \lambda_{0}(c_{0})\).  For each \(k\), pick a
  function \(f_{0,k} \in \holom[d]\) whose critical points are precisely the elements of
  \(\{z \in \lambda_{0}(c_{0}) : z \cdot c_{0} \in \mathcal{D}\}\)
  (cf.~\cite[Cor.~2.2]{forstneric2003}).  Set
  \(F_{0} : X_{0} \to \C\) by \(F_{0}(x) := f_{0,k}(\rho(\pi_{0}(x), x))\) whenever
  \(x \in \pi_{0}^{-1}(\mathcal{C}_{0,k})\).  The function \(F_{0}\) is Borel regardless of how the
  functions \(f_{0,k}\) were chosen.

  Now for the inductive step.  The main idea, as advertised earlier, is to consider a region
  \(R_{n}(c_{n})\), \(c_{n} \in \mathcal{C}_{n}\), together with its subregions
  \(R_{n-1}(c_{n-1}^{i})\), \(1 \le i \le m\), where \(c_{n-1}^{i} \in \mathcal{C}_{n-1}\). The
  function \(F_{n-1} : X_{n-1} \to \C\) is defined on each \(R_{n-1}(c_{n-1}^{i})\) separately and
  restricts to an entire function on each of these regions.  By Forstneri{\v c}'s
  Theorem~\ref{thm:forstneric}, we can find an entire function \(f : \C^{d} \to \C\) whose critical
  points in \(\lambda_{n}(c_{n})\) are exactly the elements of
  \(\{z \in \lambda_{n}(c_{n}) : z \cdot c_{n} \in \mathcal{D}\} \) and such that for
  \(F_{n}(x) = f(\rho(\pi_{n}(x),x))\) we have
  \[\sup\{|F_{n}(x) - F_{n-1}(x)| :x \in R_{n-1}(c_{n-1}^{i})\} < \epsilon_{n} \quad \text{for all } 1 \le i
    \le m.\]
  The choice of the function \(f\) depends only on the shape of regions \(R_{k}(c_{k})\), \(k < n\), inside
  \(R_{n}(c_{n})\) and on the location of critical points \(\mathcal{D}\) inside that region.  Once
  again, we have only countably many possible configurations, so we can construct the desired \(F_{n}\) using
  only countably many different functions~\(f\).

  It remains to specify \(\epsilon_{n}\) in the inductive step of the construction.  Taking
  \(\epsilon_{n} \le 2^{-n}\) ensures that \((F_{n})_{n}\) converges locally uniformly within
  orbits, and the limit function \(F\) is therefore a Borel entire function.  However, while
  the critical points of each \(F_{n}\) are precisely \(\mathcal{D}\) by construction, new critical
  points may appear in the limit\footnote{When \(d = 1\), by Hurwitz's theorem, this is equivalent to
    ensuring that \(F_{x}\) is non-constant for every \(x\).}.  So, we need to choose \(\epsilon_{n}\) to be small enough
  to ensure this does not happen.

  The choice of \(\epsilon_{n}\) is as follows.  At step \(n\), let
  \(B_{2^{-n}} \subseteq \C^{d}\) denote the closed ball of radius \(2^{-n}\) centered at the
  origin.  Define
  \[R'_{n}(c_{n}) := \{x \in R_{n}(c_{n}) : B_{2^{-n}} \cdot x \subseteq R_{n}(c_{n})\}
    \setminus (B_{2^{-n}} \cdot \mathcal{D}),\]
  i.e., the region \(R_{n}(c_{n})\) with the \(2^{-n}\)-neighborhoods of its boundary and of the
  critical points removed.  By construction, the critical points of \(F_{n-1}\) are exactly
  \(\mathcal{D}\), so for each \(c_{n-1} \in \mathcal{C}_{n-1}\),
  \[\delta_{n-1}(c_{n-1}) := \inf\{\norm{\nabla F_{n-1}(x)} : x \in R'_{n-1}(c_{n-1})\} > 0.\]
  Decreasing \(\delta_{n-1}(c_{n-1})\) if necessary, we may assume that \(\delta_{n-1}(c_{n-1}) <
  \delta_{n-2}(c_{n-2})/2 \) for any \(c_{n-2} \in \mathcal{C}_{n-2}\) such that \(R_{n-2}(c_{n-2})
  \subseteq R_{n-1}(c_{n-1})\).

  We set \(\epsilon_{n} = \epsilon_{n}(c_{n})\) to be so small that \(\epsilon_{n} \le 2^{-n}\) and
  \[\sup \{\norm{\nabla F_{n}(x) - \nabla F_{n-1}(x)} : x \in R'_{n-1}(c^{i}_{n-1})\}
    < \delta_{n-1}(c^{i}_{n-1})/2 \quad \text{for all } 1 \le i \le m.\]
  That such \(\epsilon_{n}\) exists follows from Cauchy's estimates.

  This choice of \(\epsilon_{n}\) ensures that no point outside of \(\mathcal{D}\) can be a critical
  point of \(F\).  Indeed, any \(x \notin \mathcal{D}\) satisfies \(x \in R'_{n}(\pi_{n}(x))\) for
  all sufficiently large~\(n\).  We can estimate \(\norm{\nabla F(x)}\) as follows.
  \begin{displaymath}
    \begin{aligned}
      \norm{\nabla F(x)}
      &= \lim_{m} \norm{\nabla F_{m}(x)} = \lim_{m}\norm{\nabla F_{n}(x) +
        \sum_{i=n}^{m-1}(\nabla F_{i+1}(x) - \nabla F_{i}(x))} \\
      &\ge \norm{\nabla F_{n}(x)} - \sum_{i=n}^{\infty}\norm{\nabla F_{i+1}(x) - \nabla F_{i}(x)} \\
      &\ge \delta_{n}(\pi_{n}(x)) - \sum_{i \ge n} \delta_{i}(\pi_{i}(x))/2 > 0.
    \end{aligned}
  \end{displaymath}
  We conclude that the critical points of \(F\) are exactly \(\mathcal{D}\).  Since the
  cross-section \(\mathcal{D}\) is separating, all entire functions \(F_{x}\) are distinct and the
  factor map \(x \mapsto F_{x}\) is injective.
\end{proof}

Combining Theorem~\ref{thm:injective-measurable-entire} with Theorem~\ref{thm:glucksam-weiss}, we
get the following.

\begin{corollary}[store=maincor]
  \label{cor:injective-measurable-entire}
  Every free probability measure-preserving action of \(\C^{d}\) on a standard probability
  space admits a separating measurable entire function.
\end{corollary}

It is natural to wonder whether the restriction to a subset of full measure in
Corollary~\ref{cor:injective-measurable-entire} is really necessary.  This question is closely
related to the following.

\begin{question}
  Do polynomially convex Borel toasts exist for arbitrary free Borel actions of \(\C^{d}\)?
\end{question}

\begin{remark}
  \label{rem:transitive-entire}
  The paper of Gl\"ucksam and Weiss proves more than just the existence of measurable entire
  functions.  They construct dense measurable entire functions, where all functions
  \(F_{x}\) have dense orbits inside \(\holom[d]\) under the argument shift action.  This is
  achieved by ensuring that the Borel toast has regions \(R_{n}\) with arbitrarily large ``free
  space'' on each orbit: for any compact \(K \subseteq \C^{d}\) and any \(x \in X\), there exist
  \(n\), \(c_{n} \in \mathcal{C}_{n}\), and \(z \in \C^{d}\) such that \(x E c_{n}\) and
  \[K + z \subseteq \lambda_{n}(c_{n}) \setminus \bigsqcup_{i=1}^{m}(\lambda_{n-1}(c_{n-1}^{i}) +
    \rho(c_{n}, c_{n-1}^{i})).\]
  For a given family \((p_{n})_{n}\) dense in \(\holom[d]\), the construction in~\cite{glucksam2025}
  ensures that \(F_{n}\) approximates \(p_{n}\) on a large ball \((K + z) \cdot c_{n}\) contained in
  the corresponding region \(R_{n}(c_{n})\).  This modification can be incorporated into the
  construction in Theorem~\ref{thm:injective-measurable-entire}.  In other words, if a Borel action
  \(\C^{d} \acts X\) admits a polynomially convex Borel toast with arbitrarily large free regions in
  the sense above, then one can modify the construction in
  Theorem~\ref{thm:injective-measurable-entire} to build a separating Borel entire function \(F\)
  such that the orbit of \(F_{x}\) is dense in \(\holom[d]\) for each \(x\).  A generic entire
  function has a discrete set of critical points~\cite[Cor.~8.9.3(c)]{forstneric2017}.  Therefore,
  in constructing the cross-section \(\mathcal{D}\) as in Lemma~\ref{lem:injective-cross-sections},
  we just need to add points \(c'\) such that \(\rho(c, c') \not \in \Gamma\), where \(\Gamma\) is
  the countable group generated by the values \(\rho(c,x)\) for \(c \in \mathcal{C}\) and points
  \(x\) corresponding to critical points of \(p_{n}\) added during the construction.  This allows us to
  distinguish critical points that come from the separating cross-section \(\mathcal{D}\) from those
  that come from the functions \(p_{n}\), ensuring, once again, that \(F_{x}\) has different sets of critical
  points for different \(x\).
\end{remark}

\bibliography{refs.bib}
\bibliographystyle{plain}

\end{document}